\documentclass[ 11pt]{article}

\usepackage[dvipsnames]{xcolor}
\usepackage[utf8]{inputenc}  
\usepackage[T1]{fontenc}         
\usepackage{geometry}         
\usepackage[english]{babel}  
\usepackage{cite}
\usepackage{amssymb,  xfrac, stmaryrd, blkarray, multirow, enumerate,
	verbatim,xifthen}   
\usepackage{float}
\usepackage{amsmath} 
\usepackage{amscd}
\usepackage{graphicx, color, pinlabel}    
\usepackage{amsthm}                 
\usepackage{latexsym}     
\usepackage[all]{xy}
\usepackage{hyperref}
\usepackage{fancyhdr}
\newcommand{\SLC}{{\mbox{SL}_2(\bbC)}}
\newcommand{\Tame}{{\mbox{Tame}(\mbox{SL}_2(\bbC))}}
\newcommand{\Bir}{{\mbox{Bir}(\bbP^3(\bbC))}}
\newcommand{\Birn}{{\mbox{Bir}(\bbP^n(\bbC))}}
\newcommand{\BirPlane}{{\mbox{Bir}(\bbP^2(\bbC))}}

\newcommand{\AutPlane}{{\mbox{Aut}(\bbC^2)}}

\newcommand{\bbC}{{\mathbb C}}

\newcommand{\bbN}{{\mathbb N}}

\newcommand{\bbP}{{\mathbb P}}

\newcommand{\bbR}{{\mathbb R}}

\newcommand{\bbZ}{{\mathbb Z}}

 \newcommand{\ra}{\rightarrow}

\theoremstyle{plain}      

\newtheorem{theorem}{Theorem}[section]

\newtheorem{prop}[theorem]{Proposition}

\newtheorem{thm}{Theorem}[section]

\newtheorem{lem}[theorem]{Lemma}
\newtheorem*{thmA}{Theorem A}
\newtheorem*{thmB}{Theorem B}
\newtheorem*{thmC}{Theorem C}

\makeatletter
\newdimen\bibindent
\makeatother

\newtheorem*{thm*}{Theorem}

\theoremstyle{definition} 
\newtheorem{rmk}[theorem]{Remark}
\newtheorem{definition}[theorem]{Definition}
 
 \newtheorem{example}[theorem]{Example}
 \newtheorem*{thx}{Acknowledgement}
 
\title{\textbf{On the acylindrical hyperbolicity of the tame automorphism group of SL$_2(\bbC)$}}
\author{Alexandre Martin}
\date{}
\begin{document}
\maketitle

 \begin{abstract} 
We introduce the notion of \"uber-contracting element, a strengthening of the notion of strongly contracting element which yields a particularly tractable criterion to show the acylindrical hyperbolicity, and thus a strong form of non-simplicity, of groups acting on non locally compact spaces of arbitrary dimension. We also give a simple local criterion to construct  \"uber-contracting elements for groups acting on  complexes with  unbounded links of vertices. 

As an application, we show the acylindrical hyperbolicity of the tame automorphism group of $\SLC$, a subgroup of the $3$-dimensional Cremona group $\mbox{Bir}(\bbP^3(\bbC))$, through its action on a CAT(0) square complex recently introduced by Bisi--Furter--Lamy.
 \end{abstract}

Cremona groups are groups of birationnal transformations of projective spaces over arbitrary fields, and as such are central objects in birational geometry. These groups have a long history, going back to work of Castelnouvo, Cremona,  Enriques and Noether among others. A lot of work has been devoted to Cremona groups in dimension $2$, and a rather clear picture is now available regarding the structure of such groups: Many results, such as the Tits alternative \cite{CantatCremonaTits}, the computation of their automorphism groups \cite{DesertiAutomorphismsCremona}, the Hopf property \cite{DesertiHopfCremona}, as well as their algebraic  non-simplicity \cite{CantatLamyCremonaNormalSubgroups, LonjouCremonaAcylindricallyHyperbolic}, have been proved. While classical approaches to these groups involve methods from birational geometry and dynamical systems, methods from geometric group theory have proved very fruitful in recent years to unveil more of the structure of these groups. Indeed, the group $\BirPlane$ acts by isometries on a hyperbolic space of infinite dimension, a subspace of the \textit{Picard--Manin space}, and such an action can be used to derive many more properties of the group. For instance,  Cantat--Lamy used methods from hyperbolic geometry (more precisely, ideas reminiscent of small cancellation theory) 
to show the non-simplicity of the plane Cremona group over an arbitrary closed field \cite{CantatLamyCremonaNormalSubgroups}, a result recently extended to arbitrary fields by Lonjou \cite{LonjouCremonaAcylindricallyHyperbolic}. In another direction, Minasyan--Osin used the action of the group $\AutPlane$ of polynomial automorphism group of $\bbC^2$ on the Bass--Serre tree associated to its decomposition as an amalgamated product  (a decomposition due to Jung and van der Kilk \cite{JungDecomposition, vanderKulkDecomposition}) to obtain, among other things, a strong form of non-simplicity for this group \cite{MinasyanOsinTrees}.

By contrast, the situation is much more mysterious in higher dimension, and very few results are known for Cremona groups of dimension at least $3$. A first step would be to understand subgroups of higher Cremona groups. An interesting subgroup of $\Birn$ is the automorphism groups of $\bbC^n$, or more generally the automorphism group of a space \textit{birationally equivalent} to $\bbC^n$. An even smaller subgroup is the group of \textit{tame} automorphisms of $\bbC^n$, that is, the subgroup generated by the affine group of  $\bbC^n$ and transformations of the form $(x_1, \ldots, x_n) \mapsto (x_1 + P(x_2, \ldots, x_n), \ldots, x_n)$ for some polynomial $P$ in $n-1$ variables. This definition of tame automorphisms was recently extended to a general affine quadric threefold by Lamy--V\'en\'ereau \cite{LamyVenereauTameWildThreefold}. Further methods from geometric group theory have been used recently by Bisi--Furter--Lamy to study the structure of the group $\Tame$ of tame automorphisms of $\SLC$, a subgroup of $\Bir$,  through its action on a CAT(0) square complex \cite{LamyCremonaSquareComplexes}. Such complexes have an extremely rich combinatorial geometry, and this action was used to obtain for instance the Tits alternative for the group, as well as the linearisability of its finite subgroups. \\

In this article, we explain how further methods from geometric group theory allow us to get a better understanding of the geometry and structure of $\Tame$. The aim of this article is thus twofold. On the birational geometric side, we wish to convince the reader of the general interest of the wider use of tools coming from geometric group theory in the study of Cremona groups and their subgroups. On the geometric group theoretical side, we wish to convince the reader of the interest of studying a group through its \textit{non-proper} actions on \textit{non locally} finite complexes of \textit{arbitrary} dimension. Indeed, while a lot of work has been done to study groups either through their \textit{proper} actions on metric spaces, or through their actions on simplicial \textit{trees}, few general results and techniques are available to study actions in a more general setting. Allowing non-proper actions raises serious geometric obstacles, as this most often implies working with non locally finite spaces of arbitrary dimension. We show however that, under mild geometric assumptions - assumptions that are satisfied by large classes of complexes with a reasonable combinatorial geometry, it is possible to obtain simple and useful tools to study such general actions. \\

In this article, we shall focus on the hyperbolic-like features of $\Tame$, and more precisely on the notion of \textit{acylindrical hyperbolicity}. It is a theme which has come to the forefront of geometric group theory in recent years: Indeed, it is a notion sufficiently general to encompass large classes of groups (mapping class groups \cite{BowditchTightGeodesics}, $Out(F_n)$ \cite{BestvinaFeighnOutFnAcylindricallyHyperbolic}, many CAT(0) groups \cite{SistoContractingElements}, etc.),  unifying many previously known results, and yet has strong consequences: acylindrically hyperbolic groups are SQ-universal (that is, every countable group embeds in a suitable quotient of the group),
the associated reduced $C^*$-algebra is simple in the case of a countable torsion-free group, etc. \cite{DahmaniGuirardelOsin}. We refer the reader to \cite{OsinAcylindricallyHyperbolic} for more details. It is thus non surprising to witness a wealth of tools being developed to prove the acylindrical hyperbolicity of ever larger classes of groups. 

Acylindrical hyperbolicity is defined in terms of an acylindrical action of a group on a hyperbolic space, a dynamical condition which is generally cumbersome to check, particularly for actions on non locally compact spaces: Indeed, while these conditions for actions on trees can be reformulated in terms of pointwise stabilisers of  pairs of points (this was the original definition of Sela \cite{SelaAcylindrical}), they involve \textit{coarse stabilisers} of pairs of points for more general actions, that is, group elements moving a pair of point by a given amount. Controlling such coarse stabilisers generally requires understanding the set of geodesics between two non-compact metric balls, a substantial geometric obstacle in absence of local compactness.
What is more, few actions naturally associated to a group turn out to be acylindrical (let us mention nonetheless the action of mapping class groups on their associated curve complexes \cite{BowditchTightGeodesics}, as well as the action of Higman groups on $n \geq 5$ generators on their associated CAT(-1) polygonal complexes \cite{MartinHigmanAcylindrical}). Instead, when given an action of a  group presumed to be acylindrically hyperbolic on a geodesic metric space, one generally tries to show that a sufficient criterion is satisfied, namely the existence of a WPD element with a strongly contracting orbit (see \cite[Theorem H]{BestvinaBrombergFujiwara}). This approach was for instance followed in \cite{MinasyanOsinTrees,GruberSistoAcylindricallyHyperbolic,CapraceHumeAcylindrical,BestvinaFeighnOutFnAcylindricallyHyperbolic}. Here again, checking the WPD condition for a given  element turns out to be highly non-trivial  for actions on non-locally compact spaces of dimension at least two. \\

As often in geometric group theory, the situation is much easier to handle in the case of actions on simplicial trees, that is, for groups admitting a splitting. In this case, Minasyan--Osin obtained a very simple and useful criterion \cite[Theorem 4.17]{MinasyanOsinTrees}, which was used to show the acylindrical hyperbolicity of large classes of groups: some one-relator groups, the affine Cremona group in dimension $2$, the Higman group, many $3$-manifold groups, etc. For simplicity, we only state it for amalgamated products: Consider an amalgamated product $G$ of the form $A*_C B$, where the edge group $C$ is strictly contained in both $A$ and $B$ and is \textit{weakly malnormal in $G$}, meaning that there exists an element $g\in G$ such that $C \cap gCg^{-1}$ is finite. Then $G$ is either virtually cyclic or acylindically hyperbolic. 

In order to get an idea on how one could generalise such a criterion to more general actions, it is useful to outline its proof: Considering two distinct edges $e, e'$ of the associated Bass--Serre tree with finite common stabiliser, one can extend the minimal geodesic segment $P$ containing $e$ and $e'$ into a geodesic line $L$  which is the axis of some element $g \in G$ acting hyperbolically. Such an element turns out to satisfy the WPD identity. Indeed, given two balls of radius $r$ whose projection on $L$ are sufficiently far apart with respect to $r$,  any geodesic between two points in these balls must contain a translate of $P$, and one concludes using weak malnormality. 
At the heart of this proof is the very particular geometry of trees, and in particular the existence of cut-points, which ``force'' geodesics to go through a prescribed set of vertices. Such a behaviour cannot a priori be expected from actions on higher-dimensional spaces. \\

In this article, we to show that it is possible to obtain a generalisation of the aforementioned criterion for groups acting on higher dimensional complexes, under a mild geometric assumption on the complex acted upon. Following what happens for trees, we want conditions that force large portions of a geodesic to be prescribed by a coarse information about their endpoints. In particular, we want to force geodesics to go through certain finite subcomplexes. We will be interested in the following strengthening of the notion of strongly contracting element: We consider a hyperbolic element such that one of its axes comes equipped with a set of \textit{checkpoints}, a collection of uniformly finite subcomplexes whose union is coarsely equivalent to the axis and such that for every two points of the spaces whose projections on the axis are far enough, every geodesic between them must meet sufficiently many checkpoints between their respective projections (this will be made precise in Definition \ref{def:ubercontraction}). Such elements will be called \textit{\"uber-contractions}.

The advantage of such a notion of an \"uber-contraction is that it allows for a more tractable criterion to prove the acylindrical hyperbolicity of a group. Recall that the standard criterion of  Bestvina--Bromberg--Fujiwara \cite[Theorem H]{BestvinaBrombergFujiwara} is to find a strongly contracting element satisfying the so-called WPD condition (see Theorem \ref{thm:BestvinaBrombergFujiwara}). If ones considers only \"uber-contractions, it is enough to check a weaker dynamical condition, which is much easier to check, as it is formulated purely in terms of stabilisers of pairs of points. Our main criterion is the following:

\begin{thmA}
	Let $G$ be a group acting by isometries on a geodesic metric space $X$. Let $g$ be an infinite order element which has quasi-isometrically embedded orbits,  and assume that the following holds: 
	\begin{itemize}
		\item $g$ is an \"uber-contraction with respect to a set $S$ of checkpoints,
		\item $g$ satisfies the following weakening of the WPD condition: There exists a constant $m_0$ such that for every point $s \in S$ and every $m \geq m_0$, only finitely many elements of $G$ fix pointwise $s$ and $g^ms$.
	\end{itemize}
	Then $G$ is either virtually cyclic or acylindrically hyperbolic.
\end{thmA}

With such a theorem at hand, it is now important to understand how to construct \"uber-contractions. Forcing geodesics to go through given complexes is reasonably manageable in a CAT(0) space as geodesics can be understood locally:  In a CAT(0) space, if two geodesics meet along a vertex $v$ and make a very large angle at $v$, then the concatenation of these two geodesics is again a geodesic, and what is more, any geodesic between points close enough to the endpoints of this concatenation will also have to go through $v$. For spaces which do not have such a rich geometry, we also provide a way to construct \"uber-contractions, by mimicking what happens in a space with a CAT(0) geometry: We  want to construct a hyperbolic element with an axis such that the ``angle'' made at some special vertices of this axis is so large that it will force geodesics between two arbitrary points having sufficiently far apart projections on this axis to go through these special vertices. As it turns out, a quite mild geometric condition ensures that such a ``strong concatenation of geodesics'' phenomenon occurs: We say that a complex has a \textit{bounded angle of view} if, roughly speaking, there is a uniform bound on the angle between  two arbitrary  vertices $v$ and $v'$, seen from any point that does not lie on a geodesic between $v$ and $v'$ (see Definition \ref{def:Angle_View}). Having a bounded angle of view holds for CAT(0) metric spaces and other complexes with a more combinatorial notion of non-positive curvature ($C'(1/6)$-polygonal complexes, systolic complexes). It implies in particular a Strong Concatenation Property of combinatorial geodesics (see Definition \ref{def:Strong_Concatenation_Property}), a property also satisfied by hyperbolic complexes satisfying a very weak form of isoperimetric inequality (see Proposition \ref{def:isoperimetric_inequality}).

Under such a mild assumption, we have a very simple way of constructing \"uber-contraction (see Proposition \ref{prop:Criterion_acylindrically_hyperbolic}). This is turn allows us to obtain a second, more local, criterion for acylindrical hyperbolicity which generalises the aforementioned  Minasyan--Osin criterion to actions on very general metric spaces: 

\begin{thmB}[``Link Criterion'' for acylindrical hyperbolicity]
	Let $X$ be a simply connected hyperbolic polyhedral complex satisfying the Strong Concatenation Property, together with an action by polyhedral isomorphisms of a group $G$. Assume that there exists a vertex $v$ of $X$ and a group element $g$ such that: 
	\begin{itemize}
		\item[$1)$] the action of $G_v$ on the $1$-skeleton of the link of $v$ has unbounded orbits (for the simplicial metric on the $1$-skeleton $\mbox{lk}(v)$ where every edge has length $1$),
		\item[$2)$] $G_v$ is weakly malnormal, that is, the intersection $G_v \cap gG_v g^{-1}$ is finite.
	\end{itemize}
	Then $G$ is either virtually cyclic or acylindrically hyperbolic. 
\end{thmB}

It is this criterion that we use to prove the acylindrical hyperbolicity of $\Tame$, using its aforementioned action on a CAT(0) square complex.

\begin{thmC}
	The group  $\Tame$ is acylindrically hyperbolic. In particular, it is SQ-universal and admits free normal subgroups.
\end{thmC}

The article is organised as follows. In Section $1$, after recalling standard definitions and results about acylindrical hyperbolicity, we introduce \"uber-contractions and prove Theorem A. 
In Section $2$, we introduce the Strong Concatenation Property and prove Theorem B. Finally, Section $3$ deals with the acylindrical hyperbolicity of $\mbox{Tame}(\mbox{SL}_2)$ by means of its action on the CAT(0) square complex introduced by Bisi--Furter--Lamy.

\begin{thx}
	We gratefully thank S. Lamy for  remarks on a first version of this article, as well as I. Chatterji for many discussions and suggestions about this article.
	This work was partially supported by the European Research Council (ERC) grant no. 259527 of G. Arzhantseva and by the Austrian Science Fund (FWF) grant M 1810-N25.
\end{thx}

\section{A criterion for acylindrical hyperbolicity via \"uber-contractions}

In this section, we give a tractable criterion implying acylindrical hyperbolicity for groups acting on (not necessarily locally compact)  geodesic metric spaces. 

\subsection{Contracting properties of quasi-lines}

As many others, our criterion relies on the existence of of a group element whose orbits possesses hyperbolic-like features. We start by recalling various ``contracting'' poperties of a quasi-line in a metric space.

\begin{definition}
	Let $X$ be a metric space and $\Lambda$ a quasi-line of $X$, i.e. the image by a quasi-isometric embedding of the real line. For a closed subset $Y$ of $X$, we denote by $\pi_\Lambda(Y)$ the set of points of $\Lambda$ realising the distance to $Y$, called the \textit{closest-point projection of $Y$ on $\Lambda$.}
	
	The quasi-line $\Lambda$ is \textit{Morse} if for every $K, L \geq 0$ there exists a constant $C(K,L)$ such that any $(K, L)$ quasi-geodesic with endpoints in $\Lambda$ stays in the $C(K,L)$-neighbourhood of $\Lambda$. We say that an isometry of $X$ is Morse if it is a \textit{hyperbolic} isometry, i.e. its has quasi-isometrically embedded orbits, and if one (hence every) of its orbits is Morse.
	
	The quasi-line $\Lambda$ is \textit{strongly contracting} if there exists a constant $C$ such that every ball of $X$ disjoint from $\Lambda$ has a closest-point projection on $\Lambda$ of diameter at most $C$.  We say that an isometry of $X$ is strongly contracting if it is a hyperbolic isometry and if one (hence every) of its orbits is strongly contracting.
\end{definition}

\begin{rmk}\label{rmk:strongly_contracting_Morse}
	A strongly contracting quasi-geodesic is Morse.
\end{rmk}

We now introduce a strengthening of the notion of strongly contracting isometry, which is central in this article.

\begin{definition}[system of checkpoints, \"uber-contracting isometry]\label{def:ubercontraction}
Let $X$ be a geodesic metric space, let $h$ be an isometry of $X$ with quasi-isometrically embedded orbits. A {\it system of checkpoints} for $h$ is the data of a finite subset $S$ of $X$, an \textit{error constant} $L \geq 0$, and a quasi-isometry $f: \Lambda:=\bigcup_{i \in \bbZ} h^iS \ra \bbR$ such that the following holds: 

 Let $x, y$ be points of $X$ and $x', y'$ be projections on $\Lambda$ of $x, y$ respectively. For every \textit{checkpoint} $S_i := h^i S, i \in \bbZ$ such that:
\begin{itemize}
	\item $S_i$ \textit{coarsely separates} $x'$ and $y'$ , that is, $f(x')$ and $f(y')$ lie in different unbounded connected components of $\bbR \setminus f(S_i)$,
	\item $S_i$ is at distance at least $L$ from $x'$ and $y'$,
\end{itemize}
then every geodesic between $x$ and $y$ meets $S_i$.  

A hyperbolic isometry $h$ of $X$ is \textit{\"uber-contracting}, or is an \textit{\"uber-contraction}, if it admits a system of checkpoints. 
\end{definition}

\begin{example}
If $X$ is a simplicial tree and $h$ is a hyperbolic isometry, the $h$-translates of any vertex on the axis of $h$ yields a system of checkpoints. 
	\end{example}
	
		\begin{example}
	By standard arguments of hyperbolic geometry, if $X$ is a $\delta$-hyperbolic geodesic metric space and $h$ is a hyperbolic isometry, the $h$-translates of any ball of radius $2\delta$  yields a system of checkpoints.
\end{example}

We mention a couple of immediate properties:

\begin{rmk} If $h$ is an \"uber-contraction and $\Lambda$ the $\langle h \rangle$-orbit of some finite subset, then there is coarsely well-defined closest-point projection on $\Lambda$, as the diameter of the set of  projections of a given point is uniformly bounded above.
 \end{rmk}
\begin{rmk}\label{rmk:uber_strong} An \"uber-contracting isometry is strongly contracting. In particular, it is Morse by Remark \ref{rmk:strongly_contracting_Morse}.
 \end{rmk}

\subsection{Acylindrical hyperbolicity in presence of \"uber-contractions}

We start by recalling some standard definitions.

\begin{definition}[acylindricity, acylindrically hyperbolic group]
 Let $G$ be a group acting on a geodesic metric space $X$. We say that the action is \textit{acylindrical} if for every $r \geq 0$ there exist constants $L(r), N(r) \geq 0$ such that for every points $x, y$ of $X$ at distance at least $L(r)$, there are at most $N(r)$ elements $h$ of $G$ such that $d(x, hx), d(y,hy) \leq r$.
 
 A non virtually cyclic group is \textit{acylindrically hyperbolic} if it admits an acylindrical action with unbounded orbits on a hyperbolic metric space.
\end{definition}

Our goal is to obtain a criterion for acylindrical hyperbolicity for groups admitting \"uber-contractions under additional assumptions. We start by recalling some standard criterion.

\begin{definition}
Let $G$ be a group acting on a geodesic metric space $X$.	Let $g$ be an  element of $G$ of infinite order  with quasi-isometrically embedded orbits. We say that $g$ satisfies the \textit{WPD condition} if for every $r \geq 0$ 
	and every point $x$ of $X$, there exists an integer $m_0$ such that there exists only finitely many elements $h$ of $G$ such that $d(x, hx), d(g^{m_0}x,hg^{m_0}x) \leq r$.
\end{definition}
 \begin{rmk}
  If $g$ is a Morse element, then the WPD condition is equivalent to the following strengthening, by \cite[Lemma 2.7]{SistoContractingElements}:  For every $r \geq 0$ 
  	and every point $x$ of $X$, there exists an integer $m_0$ such that for every $m \geq m_0$, there exists only finitely many elements $h$ of $G$ such that $d(x, hx), d(g^{m}x,hg^{m}x) \leq r$.
 \end{rmk}

Let us now recall a useful criterion of Bestvina--Bromberg--Fujiwara to prove the acylindrical hyperbolicity of a group: 

\begin{thm}[{\cite[Theorem H]{BestvinaBrombergFujiwara}}]\label{thm:BestvinaBrombergFujiwara}
	 Let $G$ be a group acting by isometries on a geodesic metric space $X$ and let $g$ be an infinite order element with quasi-isometrically embedded orbits. Assume that the following holds: 
	 \begin{itemize}
	 	\item $g$ is a strongly contracting element.
	 	\item $g$ satisfies the WPD condition.
	 \end{itemize}
	 Then $G$ is either virtually cyclic or acylindrically hyperbolic.
\end{thm}

We are now ready to state our main criterion: 

\begin{thm}[Criterion for acylindrical hyperbolicity]\label{prop:General_Criterion_acylindrically_hyperbolic}
	Let $G$ be a group acting by isometries on a geodesic metric space $X$. Let $g$ be an infinite order element with quasi-isometrically embedded orbits. Assume that the following holds: 
	\begin{itemize}
		\item $g$ is  \"uber-contracting (with  respect to a system of checkpoints $(h^iS)_{i \in \bbZ}$),
		\item $g$ satisfies the following weakening of the WPD condition: There exists a constant $m_0$ such that for every point $s \in S$ and every $m \geq m_0$, only finitely many elements of $G$ fix pointwise $s$ and $g^{m}s$.
	\end{itemize}
	Then $G$ is either virtually cyclic or acylindrically hyperbolic.
\end{thm}

Before starting the proof, let us recall an elementary property of coarse projections on strongly contracting quasi-lines: 

\begin{lem}[coarsely Lipschitz projection]\label{lem:coarsely_Lipschitz_projection}
 Let $\Lambda$ be a strongly contracting quasi-line, and $C$ a constant such that balls disjoint from $\gamma$ project on $\Lambda$ to subsets of diameter at most $C$. Let $x, y$ two points of $X$ and let $\pi(x), \pi(y)$ be two closest-point projections on $\Lambda$. Then $$d(\pi(x), \pi(y)) \leq \mbox{max}(C, 4d(x,y)).$$ 
\end{lem}

\begin{proof}
 Let $x, y$ be two points of $X$ and consider the ball of radius $d(x,y)$ around $x$. Then either this ball is disjoint from $\Lambda$, in which we case the strongly contracting assumption immediately implies that $d(\pi(x), \pi(y)) \leq C$, or it contains a point of $\Lambda$, in which case the distance from $x$ (respectively $y$) to any of its projection on $\Lambda$ is at most $d(x,y)$ (respectively $2d(x,y)$), and thus $d(\pi(x), \pi(y)) \leq 4d(x,y).$
\end{proof}

Before proving Proposition \ref{prop:General_Criterion_acylindrically_hyperbolic}, we present a key lemma, which reduces the proof of the WPD condition to the WPD condition for points of the checkpoints.

\begin{lem}\label{lem:key_lemma}
 Let $g$ be an \"uber-contracting element with system of checkpoints $(S_i)_{i\in \bbZ}$ and error constant $L \geq 0$. Assume that there exists a constant $m_0$ such that for every point $s \in S$ and every $m \geq m_0$, only finitely many elements of $G$ fix pointwise $s$ and $g^ms$. Then $g$ satisfies the WPD condition. 
\end{lem}

\begin{proof}

Fix $r>0$ and $x \in X$. We want to show that there exists an integer $m \geq 1$ such that the \textit{coarse stabiliser} $\mbox{Stab}_r(x, g^mx)$, that is, the set of group elements $g_i$ such that $d(x,g_ix), d(g^mx,g_ig^mx)\leq r$, is finite. 

Let $C$ be a constant such that balls disjoint from $\Lambda:= \bigcup_i S_i$ project on $\Lambda$ to subsets of diameter at most $C$.

Let 
$m:= 8r + 2C + m_0(|S|+1) + 2L$, and consider  group elements $g_i$  in $\mbox{Stab}_r(x, g^mx)$. Let $Q_x$ be a geodesic path between $x$ and $g^mx$. The constant $m$ has been chosen so that closest-point projections on $\Lambda$ of $g_ix$ and $g_ig^mx$ are at distance at least 
$m_0|g|(|S|+1)+2L$ 
by Lemma \ref{lem:coarsely_Lipschitz_projection}. 
Let $S_1, \ldots, S_k$ be the checkpoints of $\Lambda$ coarsely separating $\pi_\Lambda(x)$ and $\pi_\Lambda(g^mx)$ and which are at distance at least $L$ from $\pi_\Lambda(x)$ and $\pi_\Lambda(g^mx)$. Note that $k \geq m_0(|S|+1)$ by construction.
In particular, for each $g_i \in \mbox{Stab}_r(x, g^mx)$ there exist two distinct checkpoints of $\Lambda$, say $S$, $S' \in \{ S_1, \ldots, S_k\}$, and two points $s \in S$ and $s'\in S'$ such that $s'=g^ms$ with $m \geq m_0$, 
such that $g_iQ_x$ contains $s$ and $s'$.

This allows us to define a  map $\phi$ (using the same notation as in the previous paragraph):
 \begin{align}
\mbox{Stab}_r(x, g^mx) & \rightarrow \cup_{1 \leq j \leq k}S_j \times \cup_{1 \leq j \leq k}S_j \times Q_x \times Q_x\\
       g_i &\mapsto (s,s', g_i^{-1}s, g_i^{-1}s').
\end{align}

Notice that the target is finite. Let $F$ be the preimage of an element in the image of $\phi$. Choose an element $f_0 \in F$ and consider the set $Ff_0^{-1}$ of elements of $G$ of the form $ff_0^{-1}, f \in F$. Then elements of $Ff_0^{-1}$ fix both $s$ and $s'$ by construction. As $s'=g^ms$ with $m \geq m_0$, it follows that $Ff_0^{-1}$, and hence $F$, is finite by the weak WPD condition. It now follows that 
$\mbox{Stab}_r(x, g^mx)$ is finite.
\end{proof}

\begin{proof}[Proof of Proposition \ref{prop:General_Criterion_acylindrically_hyperbolic}]
 Let $g$ be an element of $G$ as in the statement of the Theorem and let $(S_i)_{i\in \bbZ}$ and $L \geq 0$ be a system of checkpoints as in Definition \ref{def:ubercontraction}. By \cite[Theorem H]{BestvinaBrombergFujiwara}, it is enough to show that $g$ is strongly contracting and satisfies the WPD condition. The first part follows directly from Remark \ref{rmk:uber_strong}. Since there are only finitely many elements in $\cup S_i$ modulo the action of $\langle g \rangle$, we choose $m_0 \geq 1$ such that for every point $s \in \cup S_i$ such that for every $m > m_0$, there exists only finitely many elements fixing both $s$ and $g^ms$. The second part follows directly from Lemma \ref{lem:key_lemma}
\end{proof}

\section{A local criterion}

In this section, we give a method for constructing \"uber-contractions for groups acting on polyhedral complexes with some vertices having unbounded links. This allows us to give a second, more local, criterion for proving the acylindrical hyperbolicity of such groups.

\subsection{The Strong Concantenation Property}

\begin{definition}[angle]
	Let $X$ be a polyhedral complex. The \textit{angle} at a vertex $v$ between two edges of $X$ containing $v$ is their (possibly infinite) distance in the $1$-skeleton of the link of $v$ (where each edge of that graph is given length $1$). 
	
	For two geodesic paths $\gamma, \gamma'$ starting at a vertex $v$, the \textit{angle at $v$}  between $\gamma$ and $\gamma'$, denoted $\angle_v(\gamma, \gamma')$, is the angle at $v$ between the associated edges of $\gamma$ and $\gamma'$. 
	
	Finally, for vertices $v, w, w'$ of $X$, we define the \textit{angle} $\angle_v(w,w')$ as the minimum of the angle $\angle_{v}(\gamma, \gamma')$ where $\gamma$ (respectively $\gamma'$) ranges over the combinatorial geodesics between $v$ and $w$ (respectively $w'$).
\end{definition}

The following definition mimicks the strong properties of geodesics in CAT(0) spaces.

\begin{definition}\label{def:Strong_Concatenation_Property}
We say that a complex satisfies the	Strong Concatenation Property with constants $(A, R)$ if the following two conditions hold: 
\begin{itemize}

\item Let $\gamma_1, \gamma_2$ be to geodesics of $X$ meeting at a vertex $v$. If $\angle_v(\gamma_1, \gamma_2) > A$, then $\gamma_1\cup \gamma_2$ is a geodesic of $X$.
\item Let $\gamma$ be a geodesic segment of $X$, $v$ a vertex of $\gamma$. Let $x, y$ be two vertices of $X$, $\pi(x), \pi(y)$ be projections of $x, y$ respectively on $\gamma$ such that $\pi(x)$ and $\pi(y)$ are at distance strictly more than $R$ from $v$. If $\gamma$ makes an angle greater than $A$ at $v$, then every geodesic between $x$ and $y$  contains $v$. 
\end{itemize}
\end{definition}

Note that this property has the following immediate consequence, which allows for the construction of many \"uber-contractions:

\begin{lem}[Local criterion for \"uber-contractions]\label{lem:Local_Criterion} 
	Let $X$ be a simply connected complex with 
	the Strong Concatenation Property with constants $(A,R)$. Then there exists a constant $C$ such that the following holds:
	
	Let $h$ be a hyperbolic isometry of $X$ with axis $\gamma$ and assume that for some vertex $v$ of $\gamma$, the angle made by $\gamma$ at $v$ is at least $C$. Then $h$ is \"uber-contracting with respect to the system of checkpoints $(h^iv)_{i \in \bbZ}$. \qed
\end{lem}

This allows to give a local criterion to show the acylindrical hyperbolicity of a group:

\begin{prop}[Link criterion for acylindrical hyperbolicity]\label{prop:Criterion_acylindrically_hyperbolic}
	Let $X$ be a simply connected  hyperbolic complex with the Strong Concatenation Property, together with an action by isometries of a group $G$. Assume that there exists a vertex $v$ of $X$ and a group element $g$ such that: 
	\begin{itemize}
		\item[$1)$] the action of $G_v$ on the link of $v$ has unbounded orbits (for the simplicial metric on $\mbox{lk}(v)$ where every edge has length $1$),
		\item[$2)$] the intersection $G_v \cap gG_v g^{-1}$ is finite.
	\end{itemize}
	Then $G$ is either virtually cyclic or acylindrically hyperbolic. 
\end{prop}

\begin{rmk}
	In the first item, one could  require that the action of $G_v$ on the link of $v$ has \textit{large} orbits, where large is defined in terms of constants appearing in the Strong Concatenation Property.  
\end{rmk}

\begin{proof}[Proof of Proposition \ref{prop:Criterion_acylindrically_hyperbolic}]
	Choose a vertex $v$ and a group element $g$ satisfying $1)$ and $2)$. Choose an integer $N$ greater than $8\delta$, where $\delta$ is the hyperbolicity constant of $X$. Let $P$ be a geodesic between $v $ and $gv$. By condition $1)$, choose an element $h \in G_v$ such that $P$ and $ghP$ make an angle at least $A_{}$. Then the element $gh$ is an \"uber-contraction, with $((gh)^nv)_{n \in \bbZ}$ as a system of gates. Indeed, for every $i$, $(gh)^iP$ and $(gh)^{i+1}P$ make an angle of at least $A_{}$. In particular, $\bigcup_{i \in \bbZ} (gh)^iP$ is a geodesic by the Strong Concatenation Property and thus $gh$ is a hyperbolic isometry with axis $\gamma$. Such an axis turns out to be \"uber-contracting  by Lemma \ref{lem:Local_Criterion}. 
	
	Moreover, $gh$ satisfies the so-called WPD condition by condition $2)$ and Lemma \ref{lem:key_lemma}. Hence the result follows from Proposition \ref{prop:General_Criterion_acylindrically_hyperbolic}.
\end{proof}

\subsection{Complexes with the Strong Concatenation Property}

We now give two simple properties that imply the Strong Concatenation Property. The first one is reminiscent of features of CAT(0) spaces.

\begin{definition}[bounded angle of view]\label{def:Angle_View}
 We say that a simply connected complex $X$ has an \textit{angle of view} of at most $A\geq 0$  if there exists a constant $A$ such that for every vertices $x, y$ of $X$ and every vertex $z$ of $X$ which does not lie on a geodesic between $x$ and $y$, we have $\angle_z(x,y) \leq A$. We say that $X$ has a \textit{bounded angle of view} if there exists a constant $A \geq 0$ such that $X$ has an angle of view of at most $A$.
\end{definition}

\begin{example}\label{ex:CAT0_angle}
 CAT(0) spaces have an angle of view of at most $\pi$.
\end{example}

This weak condition seems to be satisfied by many natural examples of complexes that are non-positively curved in a broad sense. For instance,  it follows from the classification of geodesic triangles by Strebel \cite{StrebelDiscDiagrams} that simply connected $C'(1/6)$ complexes in the sense of McCammond--Wise \cite{McCammondWiseFansLadders} have a bounded angle of view. The result also holds for systolic complexes.

\begin{lem}
	Simply connected complexes with an angle of view at most $A$ have the Strong Concatenation Property with constants $(3A, 0 )$.
\end{lem}

\begin{proof}
	From the bounded angle of view condition, it is immediate that the concatenation of two geodesics making an angle greater than $A$ is again a geodesic. Let us now consider the case of a geodesic segment $\gamma$ of $X$, and let $v$ be a vertex of $\gamma$. Let $x, y$ be two vertices of $X$, $\pi(x), \pi(y)$ be projections of $x, y$ respectively on $\gamma$ such that $\pi(x)$ and $\pi(y)$ are distinct from  $v$. If there exists a geodesic $\gamma_{x,y}$ between $x$ and $y$ not containing $v$, we choose geodesics $\gamma_{v,x}$, $\gamma_{v,y}$ between $v$ and $x$ (respectively between $v$ and $y$), geodesics $\gamma_x$, $\gamma_y$ between $x$ and a point $x'$ of $\pi_\gamma(x)$ (respectively between $y$ and a point$y'$  of $\pi_\gamma(y)$), and  let $\gamma_{v,x'}$, $\gamma_{v,y'}$ be the sub-segments of $\gamma$ between $v$ and $x'$ (respectively between $v$ and $y'$). By assumption, $v$ does not belong to $\gamma_x \cup \gamma_{x,y} \cup \gamma_y$, so we get 
	$$ \angle_v(\gamma) \leq \angle_v(\gamma_{v,x'}, \gamma_{v,x}) + \angle_v(\gamma_{v,x}, \gamma_{v,y}) + \angle_v(\gamma_{v,y}, \gamma_{v,y'} ) \leq 3A,$$
	which concludes the proof.
\end{proof}

\begin{rmk}
	If $X$ has a bounded angle of view, it is not necessary to require the space to be hyperbolic in the statement of Proposition \ref{prop:Criterion_acylindrically_hyperbolic}, as one can check directly that $ \bigcup_{i \in \bbZ} (gh)^iP$ is a geodesic in the proof of Proposition \ref{prop:Criterion_acylindrically_hyperbolic}.
\end{rmk}

The second property, more algorithmic in nature, is particularly suitable for hyperbolic complexes which are not known to have such a rich combinatorial geometry.

\begin{definition}\label{def:isoperimetric_inequality} We say that a simply connected polyhedral complex \textit{satisfies an isoperimetric inequality} if there exists a function $\varphi: \bbN \ra \bbN$ such that every embedded loop of length $n$ can be filled by a combinatorial disc of area at most $\varphi(n)$.
\end{definition}

\begin{lem}
 Simply connected hyperbolic complexes satisfying an isoperimetric inequality have the Strong Concatenation Property.
\end{lem}

\begin{proof}
 Let $\delta$ be the hyperbolic constant of $X$, and $\varphi$ its isoperimetric function.\\
 
 We start by the first condition. Let $\gamma_1, \gamma_2$ be two geodesics of $X$ meeting at a vertex $v$, and $\gamma$ a geodesic between the two endpoints of the path $\gamma_1 \cup \gamma_2$. For $i \in \{1, 2\}$, if $|\gamma_i|>2\delta$,  choose the point $x_i \in \gamma_i$ at distance $2\delta$ from $v$. By the hyperbolicity condition, $x_i$ is at distance at most $\delta$ from the other two sides of the geodesic triangle $\gamma_1 \cup \gamma_2 \cup \gamma$, so choose a point $u_i \in \gamma_1 \cup \gamma_2 \cup \gamma$ at distance at most $\delta$ from $v_i$, and which is not on $\gamma_i$. Choose a geodesic path $\tau_{u_i, v_i}$ between $u_i$ and $v_i$. If $|\gamma_i| \leq 2 \delta$, set $v_i = u_i$ to be the endpoint of $\gamma_i$ distinct from $v$. 
 
 If $\gamma$ does not contain $v$, then one could extract from the loop $(\gamma_1)_{v, v_1} \cup \tau_{v_1, u_1} \cup (\gamma)_{u_1, u_2} \cup  \tau_{u_2, v_2} \cup (\gamma_2)_{v_2, v}$ an embedded loop of length at most $12\delta$ containing a neighbourhood of $v$ in $\gamma_1 \cup \gamma_2$. In particular, this implies that $\angle_v(\gamma_1, \gamma_2) \leq \varphi(12\delta)$. \\
 
 Let us now show the second condition. Let $\gamma$ be a geodesic segment of $X$, $v$ a vertex of $\gamma$. Let $x, y$ be two vertices of $X$, $\pi(x), \pi(y)$ be projections of $x, y$ respectively on $\gamma$, such that $\pi(x)$ and $\pi(y)$ are at distance at least $100 \delta$ from $v$. Let us assume that there exists a geodesic $\gamma_{x,y}$ between $x$ and $y$ which does not contain $v$. 
 
 Let $\gamma'$ be the sub-segment of $\gamma$ centred at $v$ and of length $100\delta$. Using the theorem of approximation by trees in a hyperbolic space \cite[Thm. 1 Ch. 8]{CoornaertDelzantPapadopoulos} , it follows that there exists a sub-segment $\gamma_{x,y}'$ of $\gamma_{x,y}$ such that $\gamma'$ and $\gamma_{x,y}'$ are at Hausdorff distance at most $8\delta$. Projecting the endpoints of $\gamma'$ to $\gamma_{x,y}'$, it is thus possible to construct an embedded loop of $X$, of length at most $300\delta$, which contains $\gamma'$. By filling this loop, it follows that the angle of $\gamma$ at $v$ is bounded above by $\varphi(300\delta)$.\\
 
 Thus $X$ satisfies the Strong Concatenation Property with constants $( \varphi(300\delta), 100\delta)$. 
\end{proof}

\begin{rmk} 
	Note that \textit{non locally finite} hyperbolic complexes do not necessarily satisfy any isoperimetric inequality, let alone a linear one. For instance, consider the suspension of the simplicial real line, with its triangular complex structure. Given any number $n$, there exists a geodesic bigon of length $4$ between the two apices that require at least $n$ triangles to be filled. 
\end{rmk}

\section{Application: The tame automorphism group of $\SLC$}

 Here we use the Criterion \ref{prop:Criterion_acylindrically_hyperbolic} to study the subgroup $\mbox{Tame}(\mbox{SL}_2) \subset \mbox{Bir}(\bbP^3(\bbC))$. We prove the following: 
 
 \begin{thm}\label{thm:Tame_Acylindrically_Hyperbolic}
 The group  $\mbox{Tame}(\mbox{SL}_2)$ is acylindrically hyperbolic. 
\end{thm}
 
 Recall that $\mbox{Tame}(\mbox{SL}_2)$ can be defined (see \cite[Proposition 4.19]{LamyCremonaSquareComplexes}) as the subgroup  $\mbox{Tame}_q(\bbC^4)$ of $\mbox{Aut}(\bbC^4)$ generated by the orthogonal group $O(q)$ associated to the quadratic form $q(x_1,x_2,x_3,x_4) = x_1x_4 - x_2x_3$ and the subgroup  consisting of automorphisms of the form 
 $$(x_1,x_2,x_3,x_4) \mapsto (x_1, x_2+x_1P(x_1,x_3),x_3,x_4+x_3P(x_1,x_3)), P\in \bbC[x_1,x_3].$$ 
 This group was studied in \cite{LamyCremonaSquareComplexes} from the point of view of its action on a CAT(0) square complex $X$, which we now describe.\\
 
 \noindent \textbf{Vertices of $X$.} To each element $(f_1, f_2, f_3, f_4)$ of $\Tame$, one associates: 
 \begin{itemize}
 	\item a vertex $[f_1] $, said \textit{of type $1$}, corresponding to the orbit $\bbC^* \cdot f_1$,
 	\item a vertex $[f_1, f_2] $, said \textit{of type $2$}, corresponding to the orbit $\mathrm{GL}_2(\bbC) \cdot (f_1, f_2)$,
 	\item a vertex $[f_1, f_2, f_3, f_4]$, said \textit{of type $3$}, corresponding to the orbit $O(q) \cdot (f_1, f_2, f_3, f_4)$.
  \end{itemize}
  
  \noindent \textbf{Edges of $X$.} To each element $(f_1, f_2, f_3, f_4)$ of $\Tame$, one associates:
  \begin{itemize}
  	\item an edge joining the type $1$ vertex $[f_1]$ and the type $2$ vertex $[f_1, f_2]$.
  	\item  an edge joining the type $2$ vertex $[f_1, f_2]$ and the type $3$ vertex $[f_1, f_2, f_3, f_4]$.
  \end{itemize}
  
  \noindent \textbf{Squares of $X$.} To each element $(f_1, f_2, f_3, f_4)$ of $\Tame$, one associates a square with vertex set being, in cyclic order, $[f_1]$,  $[f_1, f_2]$, $[f_1, f_2, f_3, f_4]$, $[f_1, f_3]$.\\
  
  \noindent \textbf{Action of $\Tame$.} The group $\Tame$ acts by isometries on $X$ as follows. For every element $(f_1, f_2, f_3, f_4)$ of $\Tame$ and every $g$ of $\Tame$, we set: 
  \begin{itemize}
  \item $g \cdot [f_1] := [f_1 \circ g^{-1}]$,
  \item $g \cdot [f_1, f_2] := [f_1\circ g^{-1}, f_2\circ g^{-1}]$, 
  \item $g \cdot [f_1, f_2, f_3, f_4] := [f_1\circ g^{-1}, f_2\circ g^{-1}, f_3\circ g^{-1}, f_4\circ g^{-1}]$.
    \end{itemize}
  
  A central result of \cite{LamyCremonaSquareComplexes} is the following: 
  
\begin{thm}
The square complex $X$ is CAT(0) and hyperbolic. 
\end{thm}
 
We recall further properties of the action that we will need.

 \begin{prop}[{\cite[Lemma 2.7]{LamyCremonaSquareComplexes} }]\label{prop:Stab_Square}
  The action is transitive on the squares of $X$, and the pointwise stabiliser of a given square is conjugate to the subgroup 
  $$ \{ (x_1,x_2,x_3,x_4) \mapsto (ax_1, b(x_2+cx_1), b^{-1}(x_3+dx_1), a^{-1}(x_4+cx_3+dx_2)), ~~a,b,c,d \in \bbC, ab \neq 0 \}$$
  of $\mbox{Tame}(\mbox{SL}_2).$\qed
 \end{prop}

 \begin{prop}[{\cite[Propositions 3.6, 3.7 and 4.1]{LamyCremonaSquareComplexes} }]\label{prop:unbounded_orbit}
  The link of a vertex of type $1$ has infinite diameter. Moreover, the action of the stabiliser of such a vertex on its link has unbounded orbits.\qed
 \end{prop}

We want to construct a super-contraction for this action. To that end, we will use the following hyperbolic isometry considered in \cite{LamyCremonaSquareComplexes}:

\begin{lem}[{\cite[Example 6.2]{LamyCremonaSquareComplexes} }]\label{lem:Hyperbolic_Isometry}
 The element $g \in \mbox{Tame}(\mbox{SL}_2)$ defined by 
 
 $$g(x_1, x_2, x_3, x_4) = :=(x_4 + x_3x_1^2 +x_2x_1^2+x_1^5, x_2+x_1^3, x_3+x_1^3, x_1)$$ 
 acts hyperbolically on $X$. More precisely, there exists a $4\times4$ grid isometrically embedded in $X$ and a vertex $v$ of type $1$, such that the vertices $v, gv, g^2v$ are as in Figure \ref{fig:Only_Figure}.  \qed
\end{lem}

 \begin{figure}[H]
 \begin{center}
  \scalebox{1}{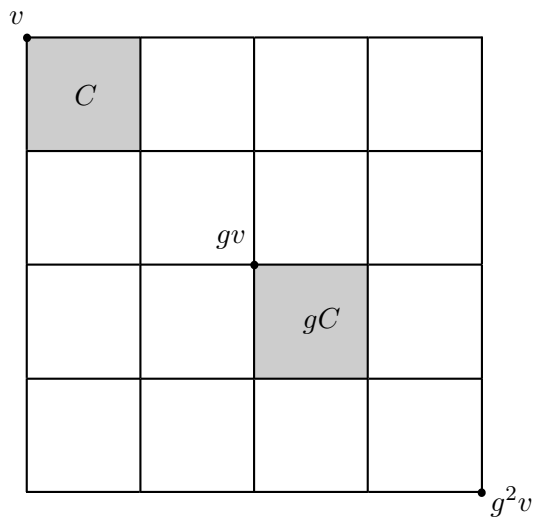}
 \caption{The $4 \times 4$ grid $K$ isometrically embedded in $X$, together with vertices $v$, $gv$ and $g^2v$ (black dots), and squares $C$ and $gC$ (shaded).}
 \label{fig:Only_Figure}
 \end{center}
 \end{figure}

\begin{proof}[Proof of Theorem \ref{thm:Tame_Acylindrically_Hyperbolic}]
 Let $v$, $K$ be the vertex and $4\times4$ grid of $X$ mentioned in Lemma \ref{lem:Hyperbolic_Isometry}. We show that we can apply the Criterion \ref{prop:Criterion_acylindrically_hyperbolic} to $g^2$ and $v$. Note that Item $1)$ of Criterion \ref{prop:Criterion_acylindrically_hyperbolic} follows from Proposition \ref{prop:unbounded_orbit}.

 We now show that  $v$ and $g^2v$ have a finite \textit{common stabiliser}, that is, that their stabilisers intersect along a finite subgroup. Let $C$ be the top-left square of $K$, as indicated in Figure \ref{fig:Only_Figure}. We start by showing that  $C$ and $gC$ have a finite common stabiliser. Indeed,  $\mbox{Stab}(C)$ is conjugated to the subgroup defined by elements of the form
 $$f(x_1,x_2,x_3,x_4)=(ax_1, b(x_2+cx_1), b^{-1}(x_3+dx_1), a^{-1}(x_4+cx_3+dx_2)), ~~a,b,c,d \in \bbC, ab \neq 0, $$ 
 by Proposition \ref{prop:Stab_Square}. Now an equation of the form $gf=f'g$, with $f,f'$ of the previous form, with coefficients $a,b,c,d$ and $a',b',c',d'$ respectively, yields the following equation, when isolating the first coordinate: 
 $$ a^{-1}(x_4 + cx_3 + dx_2)+ a^2b^{-1}x_1^2(x_3+dx_2) + a^2bx_1^2(x_2+cx_1) + a^5x^5 = a'(x_4 + x_3x_1^2 + x_2x_1^2 + x_1^5).$$
 
 Isolating the various monomials, we successively get $a^6=1, b^6 =1$ and $c=d=0$ (and analogous equalities for $a',b',c',d'$), hence $\mbox{Stab}(C) \cap g\mbox{Stab}(C)g^{-1}$ is finite, and thus $C$ and $gC$ have a finite common stabiliser.  This in turn  implies that $v$ and $g^2v$ have a finite common stabiliser. Indeed, the combinatorial interval between $v$ and $g^2v$ is exactly $K$, as combinatorial intervals embed isometrically in $\bbR^2$ with its square structure by \cite[Theorem 1.16]{PropertyACAT(0)CubeComplexes}. Thus, up taking a finite index subgroup, elements fixing $v$ and $g^2v$ will fix pointwise $K$, and in particular $C$ and $gC$.

 Since CAT(0) spaces a have bounded angle of view by Example \ref{ex:CAT0_angle}, one can thus apply  Criterion \ref{prop:Criterion_acylindrically_hyperbolic} to conclude. 
\end{proof}

\bibliographystyle{plain}
\bibliography{Higman_Cubical}

\end{document}